\documentclass[11pt]{article}
\usepackage{latexsym,amssymb,amsmath,amsthm,enumerate,geometry,cite,bbm}
\newtheorem{theorem}{Theorem}
\newtheorem{lemma}[theorem]{Lemma}

\usepackage{lineno}
\usepackage{setspace}
\allowdisplaybreaks

\begin{document}
\onehalfspace

\title{Factorially many maximum matchings\\
close to the Erd\H{o}s-Gallai bound\thanks{Research supported by research grant DIGRAPHS ANR-19-CE48-0013.}}
\author{St\'{e}phane Bessy$^1$ 
\and Johannes Pardey$^2$
\and Lucas Picasarri-Arrieta$^1$
\and Dieter Rautenbach$^2$}
\date{}

\maketitle
\vspace{-10mm}
\begin{center}
{\small 
$^1$ LIRMM, Univ Montpellier, CNRS, Montpellier, France\\
\texttt{$\{$stephane.bessy,lucas.picasarri-arrieta$\}$@lirmm.fr}\\[3mm] 
$^2$ Institute of Optimization and Operations Research, Ulm University, Ulm, Germany\\
\texttt{$\{$johannes.pardey,dieter.rautenbach$\}$@uni-ulm.de}}
\end{center}

\begin{abstract}
A classical result of Erd\H{o}s and Gallai determines the maximum size $m(n,\nu)$
of a graph $G$ of order $n$ and matching number $\nu n$.
We show that $G$ has factorially many maximum matchings 
provided that its size is sufficiently close to $m(n,\nu)$.\\[3mm]
{\bf Keywords:} Matching
\end{abstract}

\section{Introduction}

We consider finite, simple, and undirected graphs.
A {\it matching} in a graph $G$ is a set of pairwise disjoint edges,
and the {\it matching number $\nu(G)$} of $G$ is the largest size of a matching in $G$.
For a matching $M$ in $G$, let $V(M)$ be the set of vertices of $G$ that are incident with an edge in $M$;
the set $V(M)$ contains the vertices of $G$ that are {\it saturated} by $M$.

A classical result of Erd\H{o}s and Gallai, Theorem 4.1 in \cite{erga}, 
states that a graph $G$ of order $n$, size $m$, and matching number $\nu(G)$
such that $\nu(G)=\nu n$ for some $\nu\in \left[0,\frac{1}{2}\right]$ satisfies 
\begin{eqnarray}\label{e1}
m & \leq & m(n,\nu):=
\begin{cases}
\nu n(n-\nu n)+{\nu n\choose 2} & \mbox{, if $\nu \leq \frac{2}{5}-\frac{3}{5n}$, and}\\[3mm]
{2\nu n+1\choose 2} & \mbox{, if $\frac{2}{5}-\frac{3}{5n}\leq \nu\leq \frac{1}{2}$.}
\end{cases}
\end{eqnarray}
Furthermore, they showed that equality holds in (\ref{e1}) if and only if 
\begin{itemize}
\item[(i)] the complement $\overline{G}$ of $G$ is $K_{n-\nu n}\cup \overline{K_{\nu n}}$ 
for $\nu\leq \frac{2}{5}-\frac{3}{5n}$, and
\item[(ii)] $G$ is $K_{2\nu n+1}\cup \overline{K_{n-2\nu n-1}}$ for $\frac{2}{5}-\frac{3}{5n}\leq \nu\leq \frac{1}{2}$.
\end{itemize}
Recall that, for positive integers $n$ and $k$ with $k\leq n$, 
the {\it falling factorial} $n^{\underline{k}}$ is $n(n-1)\ldots (n-k+1)$.

The starting point here was the observation that the two extremal graphs in (i) and (ii) have 
$$(n-\nu n)^{\underline{\nu n}}
\,\,\,\,\,\,\,\,\,\,\mbox{ and }\,\,\,\,\,\,\,\,\,\,
\frac{(2\nu n+1)!}{(\nu n)! 2^{\nu n}}$$
maximum matchings, respectively.
Estimating quite roughly, it follows that, for positive $\nu$, the extremal graphs for (\ref{e1}) have 
between 
$\left\lceil 0.4n\right\rceil^{\underline{\left\lceil 0.5\nu n\right\rceil}}$
and 
$n^{\underline{\nu n}}$
maximum matchings.
We show that $G$ still has factorially many maximum matchings 
provided that $m(n,\nu)-m$ is sufficiently small.
Since $m(n,\nu)=\Theta(\nu n^2)$, it is natural to bound $m(n,\nu)-m$ in terms of $\nu$ and $n^2$.

The following is our first main result; all proofs are given in the next section.
\begin{theorem}\label{theorem1}
For every real $\nu$ with $\nu\in \left(0,\frac{1}{2}\right]$,
the following holds:
If $G$ is a graph of order $n$, size $m$, and matching number $\nu n$ such that 
$\left(\frac{\nu}{50}\right)^2 n\geq 1$ and $m\geq m(n,\nu)-\left(\frac{\nu}{50}\right)^2 n^2$,
then $G$ has at least 
$\left\lceil 0.1n\right\rceil^{\underline{\left\lceil 0.1\nu n\right\rceil}}$
maximum matchings.
\end{theorem}
For the sake of simplicity, we did not try to optimize the constants that appear in this statement,
which works over the full range $\left(0,\frac{1}{2}\right]$ of $\nu$.
Our purpose here is rather to illustrate the effect and present arguments and tools that allow to capture it.
In particular, the exact dependence of the minimum number of maximum matchings 
on the difference $m(n,\nu)-m$ remains a natural yet challenging open problem.

Our second main result gives a better bound provided that $\nu$ is sufficiently small.

\begin{theorem}\label{theorem2}
There are two functions 
$h_{\nu}:(0,1)\to \left(0,\frac{1}{2}\right]$ and 
$h_{\delta}:(0,1)\times \left(0,\frac{1}{2}\right]\to (0,1)$ with the following property:
If $\epsilon\in (0,1)$, 
$\nu\in \left(0,h_{\nu}(\epsilon)\right)$, and 
$G$ is a graph of order $n$, size $m$, and matching number $\nu n$ such that 
$h_{\delta}(\epsilon,\nu)n\geq 1$ and $m\geq m(n,\nu)-h_{\delta}(\epsilon,\nu)n^2$,
then $G$ has at least 
$\left\lceil (1-\epsilon)n\right\rceil^{\underline{\left\lceil (1-\epsilon)\nu n\right\rceil}}$
maximum matchings.
\end{theorem}

Matchings in graphs are among the most well studied topics in graph theory \cite{lopl},
and we would like to mention only few related results.
Computing the permanent of a matrix, and, hence, 
counting the perfect matchings of a given bipartite graph,
is a well known $\#$P-complete problem \cite{va}.
Van der Waerden's proved conjecture on the permanent of a doubly stochastic matrix \cite{eg,fa,gu,scva}
allows to show that $d$-regular bipartite graphs have exponentially many perfect matchings for $d\geq 3$,
and Br\`egman's \cite{br} upper bound on the permanent allows to derive an exponential upper bound.
Another famous related result, establishing a conjecture of Lov\'{a}sz and Plummer,
is due to Esperet, Kardo\v{s}, King, Kr\'{a}l, and Norine \cite{eskakikrno}
who showed that cubic bridgeless graphs have exponentially many perfect matchings.

\section{Auxiliary results and proofs}

Throughout this section, let $G$ be a graph of order $n$, size $m$, and matching number $\nu n$.

A key tool for our approach is the {\it Gallai-Edmonds decomposition} $D\cup A\cup C$ of a graph $G$, cf.~\cite{lopl}, where
$$
D=\{ u\in V(G):\nu(G-u)=\nu(G)\},\,\,\,\,\,\,
A=\bigcup\limits_{u\in D}N_G(u)\setminus D\mbox{, and}\,\,\,\,\,\,
C=V(G)\setminus (D\cup A),
$$
that is, the set $D$ contains the vertices that are not saturated by some maximum matching in $G$,
the set $A$ is the set of neighbors of the vertices in $D$ outside of $D$, and $C$ contains the remaining vertices.
Let the components of $G[D]$ be $G_1,\ldots,G_k$.
Each $G_i$ is {\it factor-critical}, that is,
for every vertex $u$ of $G_i$, the graph $G_i-u$ has a perfect matching.
Let $G_i$ have order $n_i$ for $i\in [k]$, $d=|D|$, $a=|A|$, and $c=|C|$.

Every maximum matching of $G$ consists of
\begin{itemize}
\item a matching of size $(n_i-1)/2$ in $G_i$ for every $i\in [k]$,
\item a matching of size $a$ in the bipartite subgraph of $G$ 
with the partite sets $A$ and $D$ formed by the edges between these two sets,
and 
\item a perfect matching in $G[C]$.
\end{itemize}
In particular, such matchings are guaranteed to exist.
Note that 
\begin{eqnarray}\label{ek}
n-2\nu n=k-a
\end{eqnarray}
is the number of vertices that are not saturated by maximum matchings in $G$.

If $G^*$ arises from $G$ by adding all missing edges 
\begin{itemize}
\item within $V(G_i)$ for each $i\in [k]$,
\item between $D$ and $A$, and 
\item within $A\cup C$,
\end{itemize}
then $G^*$ has the same Gallai-Edmonds decomposition, and, hence, also the same matching number as $G$, and 
\begin{eqnarray}\label{e2}
m(G^*)=\sum\limits_{i=1}^k {n_i\choose 2}+da+{n-d\choose 2}.
\end{eqnarray}
Using $d=n_1+\ldots+n_k$, $n_1,\ldots,n_k\geq 1$, and the convexity of $x\mapsto x^2$, we obtain
\begin{eqnarray}\label{e2b}
m(G^*) \leq m^*:={d-k+1\choose 2}+da+{n-d\choose 2}.
\end{eqnarray}
Since (\ref{e2b}) holds with equality if each but at most one component of $G^*$ is an isolated vertex,
the integer $m^*$ is the size of a graph of order $n$ and matching number $\nu n$,
and (\ref{e1}) implies $m(n,\nu)\geq m^*$.

We introduce the useful variables $x$ and $y$:
\begin{eqnarray}\label{e5}
x=\nu-\frac{a}{n}\in \left[0,\nu\right],\mbox{ and}\,\,\,\,\,\,\,\,\,\, y=\frac{d-k}{n}.
\end{eqnarray}
Note that 
\begin{eqnarray}\label{e6}
k\stackrel{(\ref{ek})}{=}n-2\nu n+a=n-2\nu n+\nu n-xn=(1-\nu-x)n,
\end{eqnarray}
and, hence, 
$$y=\frac{d-k}{n}\leq \frac{n-a-k}{n}=1-\nu+x-1+\nu+x=2x,$$
that is, $y\in \left[0,2x\right]$.

Our first lemma expresses the quadratic part of $m(n,\nu)-m^*$ in terms of $x$ and $y$.
\begin{lemma}\label{lemma-1}
$m(n,\nu)-m^*\geq g(x,y)n^2-n$ for
$$g(x,y):=
\begin{cases}
x\left(1-\nu-\frac{3}{2}x\right)+y(2x-y) & \mbox{, if $\nu \leq \frac{2}{5}-\frac{3}{5n}$ and}\\
(\nu-x)\left(\frac{5}{2}\nu+\frac{3}{2}x-1\right)+y(2x-y) & \mbox{, if $\frac{2}{5}-\frac{3}{5n}<\nu \leq \frac{1}{2}$.}
\end{cases}$$
\end{lemma}
\begin{proof}
If $\nu \leq \frac{2}{5}-\frac{3}{5n}$, then 
\begin{eqnarray*}
&& m(n,\nu)-m^*\\[3mm]
& \stackrel{(\ref{e1}),(\ref{e2b})}{=} & \left(\nu n(n-\nu n)+{\nu n\choose 2}\right)-\left({d-k+1\choose 2}+da+{n-d\choose 2}\right)\\
& = & \left(\nu n(n-\nu n)+\frac{(\nu n)^2}{2}\right)-\left(\frac{(d-k)^2}{2}+da+\frac{(n-d)^2}{2}\right)
-\underbrace{\left(\frac{\nu n}{2}+\frac{d-k}{2}-\frac{n-d}{2}\right)}_{\leq n}\\
& \geq & \left(\nu n(n-\nu n)+\frac{(\nu n)^2}{2}\right)-\left(\frac{(d-k)^2}{2}+da+\frac{(n-d)^2}{2}\right)-n\\
& \stackrel{(\ref{e5})}{=} & \left(x\left(1-\nu-\frac{3}{2}x\right)+y(2x-y)\right)n^2-n,
\end{eqnarray*}
where the last equality requires a tedious yet straightforward calculation.

Similarly, if $\frac{2}{5}-\frac{3}{5n}<\nu \leq \frac{1}{2}$, then
\begin{eqnarray*}
&& m(n,\nu)-m^*\\[3mm]
& \stackrel{(\ref{e1}),(\ref{e2b})}{=} & {2\nu n+1\choose 2}-\left({d-k+1\choose 2}+da+{n-d\choose 2}\right)\\
& = & 2(\nu n)^2-\left(\frac{(d-k)^2}{2}+da+\frac{(n-d)^2}{2}\right)
-\underbrace{\left(-\nu n+\frac{d-k}{2}-\frac{n-d}{2}\right)}_{\leq n}\\
& \geq & 2(\nu n)^2-\left(\frac{(d-k)^2}{2}+da+\frac{(n-d)^2}{2}\right)
-n\\
& \stackrel{(\ref{e5})}{=} & \left((\nu-x)\left(\frac{5}{2}\nu+\frac{3}{2}x-1\right)+y(2x-y)\right)n^2-n,
\end{eqnarray*}
where the last equality again requires a tedious yet straightforward calculation.
\end{proof}
The next lemma allows to identify the values of $x$ and $y$ for which $m(n,\nu)-m^*$ is small.

\begin{lemma}\label{lemma1}
Let $\delta$ be such that $\delta\geq \frac{3}{n}$.

If
\begin{eqnarray}\label{e3}
\left(\Big(x\geq \delta\Big)\vee \left(\nu\geq \frac{2}{5}+\frac{2}{5}\delta\right)\right)\wedge 
\left(\Big(x\leq \nu-\delta\Big)\vee \left(\nu\leq \frac{2}{5}-\frac{2}{5}\delta\right)\vee \Big(\delta\leq y\leq 2\nu-3\delta\Big)\right),
\end{eqnarray} 
then $g(x,y)\geq \delta^2$.
\end{lemma}
\begin{proof}
First, 
we assume that $\nu \leq \frac{2}{5}-\frac{3}{5n}$, which implies $x\geq \delta$.

If $x\leq \nu-\delta$, then 
$$g(x,y)=
x\left(1-\nu-\frac{3}{2}x\right)+\underbrace{y(2x-y)}_{\geq 0}
\geq \underbrace{x}_{\geq \delta}\underbrace{\left(1-\nu-\frac{3}{2}x\right)}_{\geq \frac{3}{2}\delta}
\geq \delta^2,$$
if $x>\nu-\delta$ and $\nu\leq \frac{2}{5}-\frac{2}{5}\delta$, then 
$$g(x,y)=
x\left(1-\nu-\frac{3}{2}x\right)+\underbrace{y(2x-y)}_{\geq 0}
\geq \underbrace{x}_{\geq \delta}\underbrace{\left(1-\nu-\frac{3}{2}x\right)}_{\geq \delta}
\geq \delta^2,$$
and, if $x>\nu-\delta$ and $\nu>\frac{2}{5}-\frac{2}{5}\delta$, then
$$g(x,y)=
\underbrace{x\left(1-\nu-\frac{3}{2}x\right)}_{\geq 0}+y(2x-y)
\geq \underbrace{y}_{\geq \delta}\underbrace{(2x-y)}_{\geq \delta}
\geq \delta^2.$$
Next, we assume that $\frac{2}{5}-\frac{3}{5n}<\nu\leq \frac{1}{2}$,
which implies $\nu>\frac{2}{5}-\frac{2}{5}\delta$.

Note that 
$$\frac{5}{2}\nu+\frac{3}{2}x-1\geq 
\min\left\{\frac{5}{2}\left(\frac{2}{5}-\frac{3}{5n}\right)+\frac{3}{2}\delta-1,
\frac{5}{2}\left(\frac{2}{5}+\frac{2}{5}\delta\right)-1\right\}\geq \delta.$$
If $x\leq \nu-\delta$, then 
$$g(x,y)=
(\nu-x)\left(\frac{5}{2}\nu+\frac{3}{2}x-1\right)+\underbrace{y(2x-y)}_{\geq 0}
\geq \underbrace{(\nu-x)}_{\geq \delta}\underbrace{\left(\frac{5}{2}\nu+\frac{3}{2}x-1\right)}_{\geq\delta}
\geq \delta^2,$$
and, if $x>\nu-\delta$, then 
$$g(x,y)=
\underbrace{(\nu-x)\left(\frac{5}{2}\nu+\frac{3}{2}x-1\right)}_{\geq 0}+y(2x-y)
\geq \underbrace{y}_{\geq \delta}\underbrace{(2x-y)}_{\geq \delta}
\geq \delta^2.$$
\end{proof}
The next two lemmas establish the existence of many maximum matchings in graphs 
that are close to complete bipartite graphs or complete graphs, respectively.

\begin{lemma}\label{lemma2}
For positive integers $k$ and $a$ with $k>1.1a$, 
let the bipartite graph $H$ 
arise from $K_{a,k}$ with partite sets $A$ and $K$ of orders $a$ and $k$, respectively,
by removing up to $0.08a^2$ edges in such a way that 
$H$ has a matching saturating all vertices in $A$.

The graph $H$ has at least 
$
\left\lceil k-1.1a\right\rceil^{\underline{\min\left\{\left\lceil 0.2a\right\rceil,\left\lceil k-1.1a\right\rceil\right\}}}
$
matchings saturating all vertices in $A$.
\end{lemma}
\begin{proof}
Let $M$ be a matching in $H$ saturating all vertices in $A$.
If $A$ contains a subset $A'$ of at least $0.2a$ vertices $u$ with $|N_H(u)\setminus V(M)|\geq k-1.1a$,
then there are at least 
$\left\lceil k-1.1a\right\rceil^{\underline{\min\left\{\left\lceil 0.2a\right\rceil,\left\lceil k-1.1a\right\rceil\right\}}}$
matchings that connect the vertices in $A'$ to vertices in $K\setminus V(M)$.
Since each of these matchings can be extended to a matching saturating all vertices in $A$ by using edges from $M$,
the desired statement follows.
Hence, for a contradiction, we suppose that a set $A'$ as above does not exist.
Since $H$ has at most $a^2$ edges within $V(M)$, this implies that 
\begin{eqnarray*}
m(H) & < & a^2+0.2a(k-a)+0.8a(k-a-0.1a)
=ak-0.08a^2,
\end{eqnarray*}
which is a contradiction.
\end{proof}

\begin{lemma}\label{lemma3}
For a positive integer $p$ with $p\geq 10^3$,
let the graph $K$ arise from $K_{2p}$ by removing up to $0.01{p\choose 2}$ many edges
in such a way that $K$ has a perfect matching.

The graph $K$ has at least
$\left\lceil0.447p\right\rceil !
$
perfect matchings.
\end{lemma}
\begin{proof}
Let $M=\{ x_1y_1,\ldots,x_py_p\}$ be a perfect matching in $K$.
Let the graph $H$ with vertex set $\{ z_1,\ldots,z_p\}$
be such that, for distinct indices $i$ and $j$ from $[p]$,
the two vertices $z_i$ and $z_j$ are adjacent in $H$ 
if and only if the two edges $x_iy_j$ and $x_jy_i$ belong to $K$.
Since every non-edge in $K$ leads to at most one non-edge in $H$, we have
$$m(H)\geq {p\choose 2}-0.01{p\choose 2}\geq 0.99{p\choose 2}\stackrel{p\geq 10^3}{\geq} 0.989\frac{p^2}{2},$$ 
which implies that $H$ has average degree at least $0.989p$.
This implies that $H$ contains at least $0.895p$ vertices of degree at least $0.895p$.
Hence, the graph $H$ contains a set $X$ of at least $0.447p$ vertices 
such that each vertex in $X$ has at least $0.447p$ neighbors outside of $X$.
This immediately implies that $H$ contains at least $\left\lceil 0.447p\right\rceil !$ distinct matchings.

For a matching $N$ in $H$,
let $M(N)$ arise from the perfect matching $M$ in $K$ by replacing, for every edge $z_iz_j$ in $N$,
the two edges $x_iy_i$ and $x_jy_j$ in $M$ with the two edges $x_iy_j$ and $x_jy_i$.
Clearly, the set $M(N)$ is a perfect matching in $K$, 
and $M(N)$ is distinct from $M(N')$ for distinct matchings $N$ and $N'$ in $H$.
Therefore, the graph $K$ has at least $\left\lceil 0.447p\right\rceil !$ perfect matchings.
\end{proof}

\begin{proof}[Proof of Theorem \ref{theorem1}]
Let $\nu\in \left(0,\frac{1}{2}\right]$.

Let 
\begin{eqnarray}\label{eeps}
\epsilon & = & \left(\frac{\nu}{50}\right)^2\leq 10^{-4}.
\end{eqnarray}
Let $G$ be a graph of order $n$, size $m$, and matching number $\nu n$
such that $\epsilon n\geq 1$ and $m\geq m(n,\nu)-\epsilon n^2$.
We have
\begin{eqnarray}\label{em-m}
\epsilon n^2\geq m(n,\nu)-m=
\underbrace{(m(n,\nu)-m^*)}_{\geq 0}+
\underbrace{(m^*-m(G^*))}_{\geq 0}+
\underbrace{(m(G^*)-m)}_{\geq 0},
\end{eqnarray}
where we use the notation introduced earlier in this section.

Let $\delta=2\sqrt{\epsilon}$.

Using (\ref{eeps}), $\nu\leq \frac{1}{2}$, and $n\geq \frac{1}{\epsilon}$, 
we obtain 
\begin{eqnarray}\label{edelta}
0<\delta\leq \frac{\nu}{25},\mbox{ and }\,\,\,\,\,\, 
n\geq \max\left\{\frac{3}{\delta},10^4\right\}.
\end{eqnarray}
If (\ref{e3}) holds, then Lemma \ref{lemma-1} and Lemma \ref{lemma1} imply 
$$\epsilon n^2\stackrel{(\ref{em-m})}{\geq} m(n,\nu)-m^*
\geq g(x,y)n^2-n
\geq \underbrace{\delta^2n^2}_{\geq 4\epsilon n^2}-\underbrace{n}_{\leq \epsilon n^2}
\geq 3\epsilon n^2,$$
which is a contradiction.
Hence (\ref{e3}) fails, which leads to the following three cases:
\begin{itemize}
\item {$x<\delta$ and $\nu<\frac{2}{5}+\frac{2}{5}\delta$}.
\item {$x>\nu-\delta$, $\nu>\frac{2}{5}-\frac{2}{5}\delta$, and $y<\delta$}.
\item {$x>\nu-\delta$, $\nu>\frac{2}{5}-\frac{2}{5}\delta$, and $y>2\nu-3\delta$}.
\end{itemize}
Each of these three cases corresponds to a different degeneration of the Gallai-Edmonds decomposition of $G$.
We now consider these cases separately.

\bigskip

\noindent {\bf Case 1} {\it $x<\delta$ and $\nu<\frac{2}{5}+\frac{2}{5}\delta$}.

\bigskip

\noindent Note that $\nu<\frac{2}{5}+\frac{2}{5}\delta\stackrel{(\ref{edelta})}{\leq}\frac{2}{5}+\frac{2}{125}\nu$,
which implies 
\begin{eqnarray}\label{ecase1}
\nu<0.407. 
\end{eqnarray}
By (\ref{e5}) and (\ref{e6}), we have
\begin{eqnarray}
\nu n\geq &a&=(\nu-x)n>(\nu-\delta)n\stackrel{(\ref{edelta})}{\geq} 0.96\nu n,\mbox{ and}\nonumber\\
(1-\nu)n\geq &k&=(1-\nu-x)n>(1-\nu-\delta)n\stackrel{(\ref{edelta}),(\ref{ecase1})}{>}0.576n.\label{eak}
\end{eqnarray}
Let $H_0$ be the bipartite subgraph of $G$ 
with the partite sets $A$ and $D$ formed by the edges between these two sets.
Let $M$ be a matching of size $a$ in $H_0$ that is a subset of some maximum matching in $G$.
In particular, for every $i\in [k]$, there is at most one edge in $M$ between $A$ and $V(G_i)$.
Let $H$ arise from $H_0$ by removing, for every component $G_i$ of $G[D]$, 
all but exactly one vertex in such a way that $V(M)\subseteq V(H)$.
The properties of the Gallai-Edmonds decomposition imply that every matching in $H$
that saturates all vertices in $A$ can be extended to a maximum matching in $G$.
By (\ref{em-m}), we have $m(G^*)-m\leq \epsilon n^2$, 
and, hence, 
the graph $H$ arises from $K_{a,k}$ with partite sets $A$ and $K$ of orders $a$ and $k$, respectively,
by removing up to $\epsilon n^2$ edges in such a way that $\nu(G)=a$.
Since $k\stackrel{(\ref{ecase1}),(\ref{eak})}{>}1.1a$ and 
$0.08a^2\stackrel{(\ref{eak})}{\geq}0.08\cdot 0.96^2\nu^2 n^2\stackrel{(\ref{eeps})}{\geq}\epsilon n^2$,
Lemma \ref{lemma2} implies that $H$ has at least 
$\left\lceil k-1.1a\right\rceil^{\underline{\min\left\{\left\lceil 0.2a\right\rceil,\left\lceil k-1.1a\right\rceil\right\}}}$
matchings saturating all vertices in $A$.
Since
$k-1.1a\stackrel{(\ref{ecase1}),(\ref{eak})}{>}0.1n\geq 0.1\nu n$
and
$0.2a\stackrel{(\ref{eak})}{>}0.1\nu n$,
the desired statement follows in this case.

\bigskip

\noindent {\bf Case 2} {\it $x>\nu-\delta$, $\nu>\frac{2}{5}-\frac{2}{5}\delta$, and $y<\delta$}.

\bigskip

\noindent We have $\nu>\frac{2}{5}-\frac{2}{5}\delta\stackrel{(\ref{edelta})}{\geq}\frac{2}{5}-\frac{2}{125}\nu$,
which implies 
\begin{eqnarray}\label{ecase2}
\nu>0.393. 
\end{eqnarray}
Furthermore,
\begin{eqnarray*}
|C| & = & n-d-a
=n-(d-k)-k+(\nu n-a)-\nu n\\
&\stackrel{(\ref{e5}),(\ref{e6})}{=}&\left(1-y-1+\nu+x+x-\nu\right)n\\
& = & (2x-y)n
>(2\nu-3\delta)n
\stackrel{(\ref{edelta})}{\geq}\left(2-\frac{3}{25}\right)\nu n
\stackrel{(\ref{ecase2})}{\geq} 0.738n.
\end{eqnarray*}
If $|C|=2p$, then $p\geq 0.369n\stackrel{(\ref{edelta})}{\geq}10^3$, 
and, hence, $p-1\geq 0.368n$.
Note that $G[C]$ arises from $K_{2p}$ 
by removing at most $m(G^*)-m\stackrel{(\ref{em-m})}{\leq} \epsilon n^2$ many edges.
Since 
$0.01{p\choose 2}
\geq \frac{0.01\cdot 0.368^2}{2}n^2
\stackrel{(\ref{eeps})}{\geq}\epsilon n^2$,
and every perfect matching in $G[C]$ can be extended to a maximum matching in $G$,
Lemma \ref{lemma3} implies that $G$ has at least
$\left\lceil 0.447p\right\rceil !\geq \left\lceil 0.1n\right\rceil !$ 
maximum matchings, and the desired statement follows in this case.

\bigskip

\noindent {\bf Case 3} {\it $x>\nu-\delta$, $\nu>\frac{2}{5}-\frac{2}{5}\delta$, and $y>2\nu-3\delta$}.

\bigskip

\noindent Exactly as in Case 2, we obtain (\ref{ecase2}).
Furthermore,
\begin{eqnarray*}
|D| & \geq & |D|-k
=d-k
\stackrel{(\ref{e5})}{=}yn
>(2\nu-3\delta)n
\stackrel{(\ref{edelta})}{\geq}\left(2-\frac{3}{25}\right)\nu n
\stackrel{(\ref{ecase2})}{\geq} 0.738n.
\end{eqnarray*}
Suppose, for a contradiction, that $\max\{ n_1,\ldots,n_k\}\leq 0.54n$.
Using the convexity of $x\mapsto x^2$, we obtain
\begin{eqnarray*}
\epsilon n^2 &\stackrel{(\ref{em-m})}{\geq} & m^*-m(G^*)
\stackrel{(\ref{e2}),(\ref{e2b})}{=} {d-k+1\choose 2}-\sum\limits_{i=1}^k{n_i\choose 2}
>\frac{(d-k)^2}{2}-\frac{1}{2}\sum\limits_{i=1}^k n_i^2\\
&\geq & \frac{1}{2}\left(0.738n\right)^2-\frac{1}{2}\left(\frac{n}{0.54n}\right)\left(0.54n\right)^2
\geq 0.002n^2,
\end{eqnarray*}
contradicting (\ref{eeps}). 
Hence, we may assume that $n_1\geq 0.54n\stackrel{(\ref{edelta})}{\geq}10^3$.
For $p=\frac{n_1-1}{2}$, 
we obtain $p\geq 0.269n\stackrel{(\ref{edelta})}{\geq} 10^3$ and
$p-1\stackrel{(\ref{edelta})}{\geq} 0.268n$.
Note that removing any one vertex from $G_1$, 
we obtain a graph $K$ that arises from $K_{2p}$ 
by removing at most $m(G^*)-m\stackrel{(\ref{em-m})}{\leq} \epsilon n^2$ many edges.
Furthermore, by the properties of the Gallai-Edmonds decomposition,
the graph $K$ has a perfect matching,
and every perfect matching in $K$ can be extended to a maximum matching in $G$.
Since 
$0.01{p\choose 2}
\geq \frac{0.01\cdot 0.268^2}{2}n^2
\stackrel{(\ref{eeps})}{\geq}\epsilon n^2$,
Lemma \ref{lemma3} implies that $G$ has at least
$\left\lceil 0.447p\right\rceil !\geq \left\lceil0.1n\right\rceil !$ 
maximum matchings, and the desired statement follows in this case.

This completes the proof.
\end{proof}
For the proof of Theorem \ref{theorem2}, we need 
the following variant of Lemma \ref{lemma2}.

\begin{lemma}\label{lemma2b}
Let $\epsilon,\gamma \in (0,1)$.
For positive integers $k$ and $a$ with $\frac{\epsilon k}{2}\geq a$, 
let the bipartite graph $H$ 
arise from $K_{a,k}$ with partite sets $A$ and $K$ of orders $a$ and $k$, respectively,
by removing up to $\gamma ak$ edges in such a way that 
$H$ has a matching saturating all vertices in $A$.

The graph $H$ has at least 
$
\left\lceil \left(1-\sqrt{\gamma}-\frac{\epsilon}{2}\right)k\right\rceil^{\underline{\left\lceil \left(1-\sqrt{\gamma}-\frac{\epsilon}{2}\right)a\right\rceil}}
$
matchings saturating all vertices in $A$.
\end{lemma}
\begin{proof}
Let $M$ be a matching in $H$ saturating all vertices in $A$.
If $A$ contains a subset $A'$ of at least $\left(1-\sqrt{\gamma}\right)a$ vertices 
of degree at least $\left(1-\sqrt{\gamma}\right)k$,
then each vertex $u$ in $A'$ satisfies 
$|N_H(u)\setminus V(M)|\geq \left(1-\sqrt{\gamma}\right)k-a\geq \left(1-\sqrt{\gamma}-\frac{\epsilon}{2}\right)k$,
and there are at least 
$\left\lceil \left(1-\sqrt{\gamma}-\frac{\epsilon}{2}\right)k\right\rceil^{\underline{\left\lceil \left(1-\sqrt{\gamma}-\frac{\epsilon}{2}\right)a\right\rceil}}$
matchings that connect the vertices in $A'$ to vertices in $K\setminus V(M)$.
Since each of these matchings can be extended to a matching saturating all vertices in $A$ by using edges from $M$,
the desired statement follows.
Hence, for a contradiction, we suppose that a set $A'$ as above does not exist.
This implies that 
$m(H)<\left(1-\sqrt{\gamma}\right)ak+\sqrt{\gamma}a\left(1-\sqrt{\gamma}\right)k=ak-\gamma ak$,
which is a contradiction.
\end{proof}

\begin{proof}[Proof of Theorem \ref{theorem2}]
Let $\epsilon\in (0,1)$.
Let $\nu\in \left(0,\frac{1}{2}\right]$ and $\delta\in (0,1)$ be such that 
\begin{eqnarray}\label{ehnu}
\nu &\leq & \min\left\{\frac{1}{5},\frac{\epsilon}{2\left(1+\frac{\epsilon}{2}\left(1+\frac{\epsilon}{2}\right)\right)}\right\},\,\,\,\,\,\,\mbox{ and }\,\,\,\,\,\,
\delta\leq \frac{\epsilon\nu}{8}.
\end{eqnarray}
Further restricting $\nu$ (in terms of $\epsilon$) and $\delta$ (in terms of $\epsilon$ and $\nu$), we assume that 
\begin{eqnarray}\label{ehnu2}
1-\left(1+\frac{\epsilon}{2}\right)\nu & \geq & 1-\frac{\epsilon}{2},\,\,\,\,\mbox{ and }\,\,\,\,
\left(1-\sqrt{\frac{\delta}{\left(1-\frac{\epsilon}{2}\right)\nu\left(1-\left(1+\frac{\epsilon}{2}\right)\nu\right)}}-\frac{\epsilon}{2}\right)\left(1-\frac{\epsilon}{2}\right)
\geq 1-\epsilon.
\end{eqnarray}
Let $G$ be a graph of order $n$, size $m$, and matching number $\nu n$ such that 
$\delta n\geq 1$ and $m\geq m(n,\nu)-\delta n^2$.
Using the notation introduced before in this section, Lemma \ref{lemma-1} implies
\begin{eqnarray*}
2\delta
&\geq & \frac{1}{n^2}\left(m(n,\nu)-m^*+n\right)
\geq g(x,y)
\geq x\left(1-\nu-\frac{3}{2}x\right)
\geq x\left(1-\frac{5}{2}\nu\right)
\geq \frac{x}{2}.
\end{eqnarray*}
Hence, $x\leq 4\delta$, and
\begin{eqnarray}
\nu n\geq &a&=(\nu-x)n\geq (\nu-4\delta)n\stackrel{(\ref{ehnu})}{\geq} \left(1-\frac{\epsilon}{2}\right)\nu n,\mbox{ and}\nonumber\\
&k&=(1-\nu-x)n\geq (1-\nu-4\delta)n\stackrel{(\ref{ehnu})}{\geq} \left(1-\left(1+\frac{\epsilon}{2}\right)\nu\right) n.\label{eak2}
\end{eqnarray}
Using the second upper bound on $\nu$ in (\ref{ehnu}), this implies $\frac{\epsilon}{2} k\geq a$.
Let the graph $H$ be defined as in Case 1 of the proof of Theorem \ref{theorem1}.
For $\gamma=\frac{\delta}{\left(1-\frac{\epsilon}{2}\right)\nu\left(1-\left(1+\frac{\epsilon}{2}\right)\nu\right)}$,
we obtain $\gamma ak\stackrel{(\ref{eak2})}{\geq}\delta n^2$, and 
Lemma \ref{lemma2b} implies that $H$, and, hence, also $G$, has at least 
$\left\lceil \left(1-\sqrt{\gamma}-\frac{\epsilon}{2}\right)k\right\rceil^{\underline{\left\lceil \left(1-\sqrt{\gamma}-\frac{\epsilon}{2}\right)a\right\rceil}}$
maximum matchings.
Now,
\begin{eqnarray*}
\left(1-\sqrt{\gamma}-\frac{\epsilon}{2}\right)k 
&\stackrel{(\ref{eak2})}{\geq}&\left(1-\sqrt{\gamma}-\frac{\epsilon}{2}\right)\left(1-\left(1+\frac{\epsilon}{2}\right)\nu\right) n
\stackrel{(\ref{ehnu2})}{\geq}\left(1-\sqrt{\gamma}-\frac{\epsilon}{2}\right)\left(1-\frac{\epsilon}{2}\right)n
\stackrel{(\ref{ehnu2})}{\geq} (1-\epsilon)n,\mbox{ and }\\
\left(1-\sqrt{\gamma}-\frac{\epsilon}{2}\right)a
&\stackrel{(\ref{eak2})}{\geq}& \left(1-\sqrt{\gamma}-\frac{\epsilon}{2}\right)\left(1-\frac{\epsilon}{2}\right)\nu n
\stackrel{(\ref{ehnu2})}{\geq}(1-\epsilon)\nu n,
\end{eqnarray*}
which completes the proof.
\end{proof}

\end{document}